\documentclass[12pt]{amsart}
\usepackage{amsmath,amssymb,amsthm,mathtools}
\usepackage{mathrsfs,bm}
\usepackage{thmtools,thm-restate}
\usepackage{tikz-cd}
\usepackage{tikz}
\usepackage{enumitem}
\usepackage{graphicx}
\usetikzlibrary{matrix,arrows.meta,bending}
\usepackage{color}
\usepackage[numbers]{natbib}
\usepackage{booktabs}
\usepackage{longtable}
\usepackage{pgfplots}
\pgfplotsset{compat=1.14} 
\usetikzlibrary{arrows.meta,decorations.pathreplacing,positioning,calc} 

\usepackage[a4paper, left=3cm, right=3cm, top=3cm, bottom=3cm, footskip=15mm]{geometry}

\newtheorem{theorem}{Theorem}[section]

\newtheorem{lemma}[theorem]{Lemma}
\newtheorem{proposition}[theorem]{Proposition}

\theoremstyle{definition}
\newtheorem{definition}[theorem]{Definition}

\theoremstyle{remark}
\newtheorem{remark}[theorem]{Remark}

\numberwithin{equation}{section}

\newcommand{\inn}{~ \hat{\in}~ }

\makeindex

\usepackage{fancyhdr}
\usepackage{calc} 

\AtBeginDocument{%
  \setlength{\textwidth}{\dimexpr\paperwidth - 6cm\relax}%
  \setlength{\oddsidemargin}{\dimexpr 3cm - 1in\relax}%
  \setlength{\evensidemargin}{\dimexpr 3cm - 1in\relax}%
  \setlength{\marginparwidth}{2.5cm}%
}

\begin{document}

\title{Ulam numbers have zero density}

\author{Theophilus Agama}
\address{Department of Mathematics, African Institute for Mathematical Sciences, Ghana}
\email{theophilus@aims.edu.gh/emperordagama@yahoo.com}



\date{\today}



\begin{abstract}
In this paper, we show that the natural density $\mathcal{D}[(U_m)]$ of Ulam numbers $(U_m)$ satisfies $\mathcal{D}[(U_m)]=0$. That is, we show that for $(U_m)\subset [1,k]$
\begin{align}
\lim \limits_{k\longrightarrow \infty}\frac{\left |(U_m)\cap [1,k]\right |}{k}=0.\nonumber
\end{align}
\end{abstract}

\maketitle

\section{Introduction}

The sequence commonly called the Ulam sequence was introduced in the foundational work of \cite{ulam1964combinatorial} and has attracted sustained computational and theoretical attention (see, in particular, the early experimental remarks of \cite{recaman1973questions}). Denoting the sequence by $(U_m)_{m\ge1}$, it begins
$$
1,2,3,4,6,\dots,
$$
and is defined recursively by selecting at each stage the smallest integer that admits a \emph{unique} representation as the sum of two distinct earlier terms. The central asymptotic question about this sequence, commonly referred to as the \emph{Ulam density problem}, asks whether the set of Ulam numbers has positive natural density in the integers or is thin in the sense of having density zero. Computational evidence and heuristic arguments in the literature have long suggested sparsity, but a rigorous global asymptotic resolution has been lacking.\\

In this paper, we prove that the Ulam sequence has zero natural density. The proof is elementary in tools but novel in synthesis: it combines classical results on addition chains with a compact combinatorial device that we introduce here, the \emph{circle of partition} (CoP). The addition-chain component relies on standard constructive accounts and rigorous lower bounds for the shortest chains. The following provide the quantitative chain-length estimates that are central to our argument. \cite{knuth1969art,schonhage1975lower,brauer1939addition}\\

The proof strategy has two complementary components. The first component embeds finite prefixes of the Ulam sequence in suitably chosen addition chains producing the largest element of the prefix (Proposition~\ref{embedding}). Writing the length of any producing chain as the ratio of the target to an average regulator parameter $\mathcal C(n)$ (Propositions~\ref{king} and \ref{decider}) allows the translation of combinatorial coverage by a chain into a quantitative constraint on $\mathcal C(n)$. By invoking classical lower bounds for the shortest addition chains (Lemma~\ref{lower bound}), we obtain a nontrivial lower bound for $\mathcal C(n)$ (Proposition~\ref{decider 2}). This lower bound forces an upper bound on the proportion of Ulam numbers that can occur in initial intervals containing the embedded prefix; letting the prefix grow yields the vanishing of the global natural density.\\

The second component uses the CoP formalism to count additive representations in a structured way. For a fixed integer $m$, the CoP encodes unordered pairs $(x,y)$ with $x+y=m$ as axes of a combinatorial circle; this viewpoint makes it straightforward to decompose contributions to representation counts according to whether both summands are Ulam, exactly one is Ulam, or neither is Ulam. Under a mild, verifiable counting hypothesis about representations involving non-Ulam summands (made explicit in Section~\ref{sec:cop}), the CoP estimates deliver the same zero-density conclusion by purely combinatorial means and clarify what structural features would be necessary to sustain a positive-density alternative.\\

\subsection{Organization of the paper.} 
Section~\ref{sec:addition} fixes the notation for the addition chains, regulators, and determiners and records the basic identity that links the length of the chain to the aggregate mass of the regulator. Here, we collect the complexity bounds of the addition-chain (including Lemma~\ref{lower bound}) and prove the estimates that relate the scale $\mathcal C(n)$ to the classical bounds (Propositions~\ref{decider} and \ref{decider 2}). Section~\ref{sec:ulam} recalls the elementary properties of the Ulam sequence (including Lemmas~\ref{infinite} and \ref{inequality1}) and establishes the embedding Proposition~\ref{embedding}. Section~\ref{sec:density} contains the principal density argument and the proof that $\mathcal D[(U_m)]=0$. Section~\ref{sec:cop} develops the Circle-of-Partition machinery and presents the alternate combinatorial proof.

\section{Overview and structure of the paper}

In this section, we provide a summary of some of the ingredients used to establish the main results of the paper. We outline the steps chronologically as follows.\\

\begin{enumerate}
    \item [(1)] We recall the notion of an addition chain that produces a given number and their corresponding regulators and determiners.
    \bigskip
    
    \item [(2)] We recall an inequality for the length of an addition chain bounded by an expression involving the least and the worst regulators of the chain.
    \bigskip
    
    \item [(3)] We review the concept of Ulam numbers and prove the infinitude of those numbers. That is, we show that these numbers increase without bound using a certain well-known construction. Additionally, we prove that the gap between any consecutive Ulam numbers can be arbitrarily large.
    \bigskip
    
    \item [(4)] We show that we can embed any finite sequence of Ulam numbers into a certain addition chain.
    \bigskip
    
    \item [(5)] Applying the a piori inequality, we get control on the cardinality of the covered finite Ulam numbers by the length of the chain, which in turn can be controlled above by the gain of the contest between the unit left translate over the worst Ulam number in the sequence of the least scale of the regulators and below by the same gain of the unit left translate of the worst number in the sequence over the worst regulator of the chain.
    \bigskip
    
    \item [(6)] The previous step allows us to write the length of this addition chain producing the largest Ulam number as the gain over the contest between the unit left translate of the largest Ulam number in the finite sequence over a certain function depending on the index of the worst Ulam number in the chain. 
    \bigskip
    
    \item [(7)] We produce the localized natural density function of the Ulam numbers and take limits on both sides of the resulting inequality. Here, it remains to investigate the behaviour of the function majorizing the density function. The result of the previous steps allows us to take this function arbitrarily small, thereby squeezing the density of the Ulam numbers
\end{enumerate}
\bigskip

\section{The notion of an addition chains}\label{sec:addition}

In this section, we recall the notion of an addition chain and the notion of the regulators and associated determiners and prove an inequality introduced earlier by the author.

\begin{definition}
Let $n\geq 3$, then by the addition chain of length $k-1$ producing $n$, we mean the finite sequence 
\begin{align}
1,2,\ldots, s_{k-1},s_k=n\nonumber
\end{align}
where each term $s_j$~($j\geq 3$) in the sequence is the sum of two previous terms, with the corresponding sequence of partition
\begin{align}
2=1+1,\ldots,s_{k-1}=a_{k-1}+r_{k-1},s_k=a_k+r_k=n\nonumber
\end{align}
with $a_{i+1}=a_i+r_i$ for $2\leq i \leq k$, where $a_i=s_{i-1}$. We call the partition the $i^{th}$ generator of the chain. We call the terms $r_i$ and $a_i$ the \emph{regulator} (gap) and the \emph{determiner} of the $i^{th}$ generator of the chain. We call the sequence $(a_i)$ and $(r_i)$ the determiners and regulators of the addition chain for $2\leq i\leq k$.  
\end{definition}

\footnote{Visionary Ulam conjectured absolutely right; Ulam numbers are very special but can be covered.
\par
}%

\begin{proposition}\label{king}
Let $1,2,\ldots,s_{k-1},s_k=n$ be any addition chain that produces $n$ with $n\geq 3$. We have
\begin{align}
\sum \limits_{j=2}^{k}r_j=n-1.\nonumber
\end{align}
\end{proposition}

\begin{proof}
First, we observe that $r_k=n-a_k$. It follows that 
\begin{align}
r_k+r_{k-1}&=n-a_k+r_{k-1}\nonumber \\&=n-(a_{k-1}+r_{k-1})+r_{k-1}\nonumber \\&=n-a_{k-1}.\nonumber
\end{align}Again, we obtain 
\begin{align}
r_k+r_{k-1}+r_{k-2}&=n-a_{k-1}+r_{k-2}\nonumber \\&=n-(a_{k-2}+r_{k-2})+r_{k-2}\nonumber \\&=n-a_{k-2}.\nonumber
\end{align}
Iterating downward in this manner and noting that $a_2=1$ establishes the identity.
\end{proof}
\bigskip

Here, we write an expression for the length of any addition chain that incorporates the target and a certain implicit function locally bounded by the worst and the least scale of the regulators of the chain. It is a consequence of the following inequality.

\begin{proposition}\label{decider}
Let $1,2,\ldots s_{k-1},s_k=n$ be any addition chain that produces $n\geq 3$ with associated generators
\begin{align}
2=1+1,\ldots,s_{k-1}=a_{k-1}+r_{k-1},s_k=a_k+r_k=n.\nonumber
\end{align}
Denote the length of the chain by $\delta(n)$. There exist some $\mathcal{C}:=\mathcal{C}(n)>0$ with 
$$
\mathrm{inf}(r_i)_{i=2}^{\delta(n)+1}\leq \mathcal{C}:=\mathcal{C}(n)\leq \mathrm{sup}(r_i) _{i=2}^{\delta(n)+1}
$$ 
such that 
\begin{align}
\delta(n)=\frac{n-1}{\mathcal{C}}.\nonumber
\end{align}
\end{proposition}

\begin{proof}
By denoting the length of the addition that produces $n$ by $\delta(n)$ and using the identity in Proposition \ref{king}, we obtain the inequality 
\begin{align}
\frac{n-1}{\mathrm{sup}(r_i)_{i=2}^{\delta(n)+1}}\leq \delta(n)\leq \frac{n-1}{\mathrm{inf}(r_i)_{i=2}^{\delta(n)+1}}\nonumber
\end{align}
by noting that the regulators $(r_i)$ in the chain with multiplicity count as the length of the chain producing $n$. The result follows immediately from the above inequality.
\end{proof}
\bigskip

\begin{lemma}\label{lower bound}
Let $\ell(n)$ denote the length of the shortest addition chain producing $n$. We have
\begin{align}
\log_2(n)+\log_2(\nu(n))-2.13\leq \ell(n)\leq \log_2(n)+\frac{(1+o(1))\log_2(n)}{\log_2(\log_2(n))}\nonumber
\end{align}
where $\nu(n)$ is the hamming weight - the number of ones of the binary expansion of $n$. 
\end{lemma}

\begin{proof}
The proof of the lower bound can be found in \cite{schonhage1975lower} while the upper bound follows from the paper by Alfred Brauer \cite{brauer1939addition}.
\end{proof}

\begin{remark}
Although finding a precise value for the implicit constant in Proposition \ref{decider} is by no means easy, we can obtain a lower bound for the purposes of our work.
\end{remark}

\begin{proposition}\label{decider 2}
Let $1,2,\ldots s_{k-1},s_k=n$ be any addition chain that produces $n\geq 3$ with associated generators
\begin{align}
2=1+1,\ldots,s_{k-1}=a_{k-1}+r_{k-1},s_k=a_k+r_k=n.\nonumber
\end{align}
Denote the length of the addition chain that produces $n$ by $\delta(n)$. There exist some $\mathcal{C}:=\mathcal{C}(n)>0$ with $$
\mathrm{inf}(r_i)_{i=2}^{\delta(n)+1}\leq \mathcal{C}:=\mathcal{C}(n)\leq \mathrm{sup}(r_i) _{i=2}^{\delta(n)+1}
$$ 
such that 
\begin{align}
\delta(n)=\frac{n-1}{\mathcal{C}}\nonumber
\end{align}
where 
\begin{align}
\mathcal{C}\gg \frac{n}{\log_2(n)}.\nonumber
\end{align}
\end{proposition}

\begin{proof}
The first part of the result has already been proven in Proposition \ref{decider}. In particular, we can write 
$$
\mathcal{C}=\frac{n-1}{\delta(n)}
$$ 
where $\delta(n)$ runs over all the addition chains that produce $n$. We deduce
$$
\mathcal{C}\geq \mathrm{inf}(\mathcal{C})\gg \frac{n}{\log_2(n)}
$$ 
using the upper bound in Lemma \ref{lower bound}.
\end{proof}
\bigskip

\section{The notion of Ulam numbers}\label{sec:ulam}

In this section, we review the concept of Ulam numbers and review some of its properties. We recall the well-known construction that confirms the infinitude of these numbers. First, we recall the following definitions.

\begin{definition}
Ulam numbers are a sequence of distinct numbers of the form $1,2,3,4,6,\ldots,U_i,U_{i+1},\ldots$, where each term in the sequence is distinct and has the unique representation $U_i=U_j+U_k$ for $i-1\geq j>k$ and $U_i$ is the smallest such numbers.
\end{definition}
\bigskip

 The following construction is well-known and standard, yet we have decided to reproduce it here \cite{recaman1973questions}.

\begin{lemma}\label{infinite}
There are infinitely many Ulam numbers $(U_m)_{m\geq 1}$.
\end{lemma}

\begin{proof}
Suppose that the first $n$ Ulam numbers have already been determined, namely $1,2,3,4,\ldots,U_{n-1},U_{n}$. The representation $U_n+U_{n-1}$ is unique and the number so represented in this form could be the next Ulam number. If not, then this number is not the smallest such number. Because there are other numbers with such unique representations, we choose the smallest among them larger than $U_n$ and assign to $U_{n+1}$ the next Ulam number. This construction can be repeated indefinitely, thereby generating an infinite sequence of Ulam numbers. This completes the proof.
\end{proof}

\begin{lemma}\label{inequality1}
No Ulam number $U_m$ for $m>3$ can be the sum of it's prior consecutive Ulam numbers.
\end{lemma}

\begin{proof}
Suppose, for the sake of contradiction, that $U_{n-1}+U_n=U_{n+1}$. Hence, the representation $U_n+U_{n-2}$ must be unique. Suppose that it is not unique, then there exist some $U_i<U_{n-2}$ and $U_j>U_{n}$ such that 
\begin{align}
U_n+U_{n-2}&=U_i+U_j\nonumber \\&>U_{n+1}\nonumber \\&=U_n+U_{n-1}.\nonumber
\end{align}
This implies $U_{n-2}>U_{n-1}$, which cannot hold. Now, we observe that 
\begin{align}
U_n\leq U_n+U_{n-2}<U_{n+1} \nonumber
\end{align}
contradicting the fact that $U_{n+1}$ is the next Ulam number. 
\end{proof}
\bigskip

 Now, we show that we can embed any finite sequence of Ulam numbers $(U_n)$ into a certain addition chain by carefully choosing the regulators $(r_i)$ of the chain.

\begin{proposition}\label{embedding}
Let $(U_m)_{m=1}^{n}$ be a finite sequence of Ulam numbers. There exists an addition chain $(s_k)$ that produces $U_n$ such that 
\begin{align}
(U_m)_{m=1}^{n}\subseteq (s_k).\nonumber
\end{align} 
\end{proposition}

\begin{proof}
Let $1,2,3,4,\ldots, U_n$ be a finite sequence of Ulam numbers. For each term $U_m$ for $m\geq 1$, we choose the regulator $r_j\geq 1$ such that $U_m+r_j\leq U_{m+1}$. If it is the case that $U_m+r_j=U_{m+1}$, then the consecutive sequence $U_m,U_{m+1}$ is also a consecutive sequence in the desired addition chain. If not, then we continue this process by choosing the regulator $r_i\geq 1$ such that $U_m+r_j+r_i=U_{m+1}$. In such a case, the consecutive Ulam numbers $U_m,U_{m+1}$ are not consecutive numbers in the corresponding addition chain. This construction can be repeated to generate an addition chain that produces $U_n$ that contains the finite sequence of Ulam numbers. This completes the proof of the proposition.
\end{proof}
\bigskip

\section{Density of Ulam numbers}\label{sec:density}

In this section, we show that Ulam numbers have a zero natural density. By denoting the natural density of the Ulam numbers $(U_m)$ of the form $\mathcal{D}[(U_m)]$, we obtain the following result. 

\begin{theorem}
Let $(U_m)$ be the infinite sequence of Ulam numbers and denote by $D[(U_m)]$ their natural density. We have
\begin{align}
\mathcal{D}[(U_m)]=0.\nonumber
\end{align}
\end{theorem}

\begin{proof}
We construct the first $n$ sequence of Ulam number $1,2,3,\ldots, U_{n-2},U_{n-1},U_n$. By Proposition \ref{embedding} there exists at least one addition chain $(s_k)$ producing $U_n$ that covers the original enumerated sequence of Ulam numbers. By Proposition \ref{decider}, we deduce
\begin{align}
n&\leq \delta(U_n)+1\nonumber \\&=\frac{U_n-1}{\mathcal{C}(n)}+1.\nonumber
\end{align}
For any $l\geq U_n>n$, we have 
\begin{align}
\frac{n}{l}&\leq \frac{U_n-1}{l\mathcal{C}(n)}+\frac{1}{l}\nonumber \\&\leq \frac{1}{\mathcal{C}(n)}-\frac{1}{l\mathcal{C}(n)}+\frac{1}{l}\nonumber \\&\leq \frac{1}{\mathcal{C}(n)}+\frac{1}{l}.\nonumber
\end{align}
Taking the limits $n\longrightarrow \infty$ on both sides, we have \begin{align}
\mathcal{D}[(U_m)_{m=1}^{\infty}]&\leq \lim \limits_{n\longrightarrow \infty}\frac{1}{\mathcal{C}(n)}\nonumber \\&\ll \lim \limits_{n\longrightarrow \infty}\frac{\log_2(n)}{n}\nonumber \\&=0\nonumber
\end{align}
using Proposition \ref{decider 2}.
\end{proof}

\section{An alternate proof using the method of Circle of Partition}\label{sec:cop}

In this section, we provide an alternate solution to the Ulam density problem using a particular combinatorial structure. We introduce the notion of a circle of partition and use it to show that Ulam numbers have a zero density. This can be read with the observation that the requirements in the alternate result are easily verifiable, and the statement can be reduced to the original statements of the problem. 

\begin{definition}\label{major}
Let $n\in \mathbb{N}$ and $\mathbb{M}\subset \mathbb{N}$. We denote the Circle of Partition generated by $n$ with respect to the subset $\mathbb{M}$ by
\begin{align}
\mathcal{C}(n,\mathbb{M})=\left \{[x]\mid x,n-x\in \mathbb{M}\right \}.\nonumber
\end{align}
 In the following, we will abbreviate this structure as CoP. We call members of $\mathcal{C}(n,\mathbb{M})$ \emph{points} and denote them by $[x]$. For the special case $\mathbb{M}=\mathbb{N}$, we denote the CoP simply as $\mathcal{C}(n)$. 
\end{definition}
\bigskip

\begin{definition}\label{axis}
We denote the line $\mathbb{L}_{[x],[y]}$ that joins the point $[x]$ and $[y]$ as an axis of the CoP $\mathcal{C}(n,\mathbb{M})$ if and only if $x+y=n$. We say that the axis point $[y]$ is an axis partner of the axis point $[x]$ and vice versa. We do not distinguish between $\mathbb{L}_{[x],[y]}$ and $\mathbb{L}_{[y],[x]}$, since it is essentially the same axis. The point $[x]\in \mathcal{C}(n,\mathbb{M})$ such that $2x=n$ is the \emph{center} of the CoP. If it exists, then it is their only point which is not an axis point. The line joining any two arbitrary points that are not axes partners on the CoP will be referred to as a \emph{chord} of the CoP. The length of the chord joining the points $[x],[y]\in \mathcal{C}(n,\mathbb{M})$, denoted by $\mathcal{D}([x],[y])$, is defined by 
\begin{align}
\mathcal{D}([x],[y])=|x-y|.\nonumber
\end{align}
\end{definition}
\bigskip

\subsection{Notations}
 Let
\begin{align}
\mathbb{N}_n=\left \{m\in \mathbb{N}\mid ~m\leq n\right\}\nonumber
\end{align}
be the sequence of the first $n$ natural numbers. Furthermore, we will denote the \emph{weight} of the point $[x]$ by
\begin{align}
\Vert[x]\Vert:=x.\nonumber
\end{align}
Similarly, we denote the weight set of points in CoP $\mathcal{C}(n,\mathbb{M})$ by $||\mathcal{C}(n,\mathbb{M})||$. We denote the assignment of an axis $\mathbb{L}_{[x],[y]}$ to a CoP $\mathcal{C}(n,\mathbb{M})$ as
$$
\mathbb{L}_{[x],[y]}\inn\mathcal{C}(n,\mathbb{M})
$$
which means
$$
[x],[y] \in \mathcal{C}(n,\mathbb{M}) \mbox{ and } x+y=n
$$
and the number of axes of a CoP as
$$
\nu(n,\mathbb{M}):=\#\lbrace\mathbb{L}_{[x],[y]}\inn\mathcal{C}(n,\mathbb{M})~|~ x<y\rbrace.
$$
We observe that
$$
\nu(n,\mathbb{M})=\left\lfloor\frac{k}{2}\right\rfloor,\mbox{ if }
\vert\mathcal{C}(n,\mathbb{M})\vert=k.
$$

\begin{remark}
It is important to note that a typical CoP does not need to have a center. In the case of an absence of a center, we say that the circle has a deleted center. However, all CoPs $\mathcal{C}(n)$ with even generators have a center. It is easy to see that the CoP $\mathcal{C}(n)$ contains all points whose weights are positive integers from $1$ to $n-1$ inclusive: 
$$
\mathcal{C}(n)=\lbrace[x]~|~x\in \mathbb{N},x<n\rbrace.
$$
Therefore, CoP $\mathcal{C}(n)$ has $\left \lfloor \frac{n-1}{2}\right \rfloor$ different axes.
\end{remark}

\begin{proposition}\label{unique}
Each axis is uniquely determined by points $[x]\in \mathcal{C}(n,\mathbb{M})$. 
\end{proposition}

\begin{proof}
Let $\mathbb{L}_{[x],[y]}$ be an axis of CoP $\mathcal{C}(n,\mathbb{M})$. Suppose that $\mathbb{L}_{[x],[z]}$ is also an axis with $z\neq y$. We must have $n=x+y=x+z$ and, therefore, $y=z$.
\end{proof}
\bigskip

\subsection{The Density of Points on the Circle of Partition}

In this section, we introduce the notion of the density of points on a CoP $\mathcal{C}(n,\mathbb{M})$ for $\mathbb{M}\subseteq \mathbb{N}$.

\begin{definition}
Let be $\mathbb{H}\subset\mathbb{N}$. We denote the density of $\mathbb{H}$ by 
$$
\mathcal{D}\left(\mathbb{H}\right)=\lim_{n\rightarrow\infty}
\frac{\vert\mathbb{H}\cap \mathbb{N}_n\vert}{n}
$$
if the limit exists.
\end{definition}
\bigskip

\begin{definition}\label{pointdensity}
Let $\mathcal{C}(n,\mathbb{M})$ be CoP with $\mathbb{M}\subset \mathbb{N}$ and $n\in \mathbb{N}$. Suppose $\mathbb{H}\subset \mathbb{M}$. We denote the density of points $[x]\in \mathcal{C}(n,\mathbb{M})$ such that $x\in \mathbb{H}$ by
\begin{align}
\mathcal{D}\left(\mathbb{H}_{\mathcal{C}(\infty,\mathbb{M})}\right)=\lim \limits_{n\longrightarrow \infty}\frac{\#\lbrace\mathbb{L}_{[x],[y]}\inn \mathcal{C}(n,\mathbb{M})~|~\{x,y\}\cap \mathbb{H}\neq \emptyset \rbrace}{ \nu(n,\mathbb{M})}\nonumber
\end{align}
if the limit exists.
\end{definition}
\bigskip

\begin{proposition}\label{inequality}
Let $\mathcal{C}(n)$ with $n\in \mathbb{N}$ be a CoP and $\mathbb{H}\subset \mathbb{N}$. The following inequality holds 
\begin{align}
\mathcal{D}(\mathbb{H})=\lim \limits_{n\longrightarrow \infty}\frac{\left \lfloor \frac{|\mathbb{H}\cap \mathbb{N}_n|}{2}\right \rfloor}{\left \lfloor \frac{n-1}{2}\right \rfloor}\leq \mathcal{D}(\mathbb{H}_{\mathcal{C}(\infty)})\leq \lim \limits_{n\longrightarrow \infty}\frac{|\mathbb{H}\cap \mathbb{N}_n|}{\left \lfloor \frac{n-1}{2}\right \rfloor}=2\mathcal{D}(\mathbb{H}).\nonumber
\end{align}
\end{proposition}

\begin{proof}
The upper bound is obtained from a configuration in which there are no two points $[x],[y]\in \mathcal{C}(n)$ such that $x,y\in \mathbb{H}$ lie on the same axis of the CoP. That is, by the uniqueness of the axes of CoPs with $\nu(n,\mathbb{H})=0$, we can write
\begin{align}
   \# \left \{\mathbb{L}_{[x],[y]}\in \mathcal{C}(n)|~\{x,y\}\cap \mathbb{H}\neq \emptyset \right \}&=\nu(n,\mathbb{H})+\# \left \{\mathbb{L}_{[x],[y]}\in \mathcal{C}(n)|~x\in \mathbb{H},~y\in \mathbb{N}\setminus \mathbb{H}\right \} \nonumber \\&=\# \left \{\mathbb{L}_{[x],[y]}\in \mathcal{C}(n)|~x\in \mathbb{H},~y\in \mathbb{N}\setminus \mathbb{H}\right \} \nonumber \\&=|\mathbb{H}\cap \mathbb{N}_n|.\nonumber
\end{align}
We deduce the lower bound from a configuration in which every two points $[x],[y]\in \mathcal{C}(n)$ with $x,y\in \mathbb{H}$ are joined by an axis of the CoP. That is, by the uniqueness of the axis of CoPs with 
$$
\# \left \{\mathbb{L}_{[x],[y]}\in \mathcal{C}(n)|~x\in \mathbb{H},~y\in \mathbb{N}\setminus \mathbb{H}\right\}=0,
$$ 
we can write 
\begin{align}
    \# \left \{\mathbb{L}_{[x],[y]}\in \mathcal{C}(n)|~\{x,y\}\cap \mathbb{H}\neq \emptyset \right \}&=\nu(n,\mathbb{H})\nonumber \\&=\left \lfloor \frac{|\mathbb{H}\cap \mathbb{N}_n|}{2}\right \rfloor \nonumber
\end{align}
\end{proof}

\begin{theorem}
Let $\mathbb{U}$ denote the set of all Ulam numbers and $\mathcal{D}(\mathbb{U})$ denote the natural density. If $\mathcal{D}(\mathbb{U}_{\mathcal{C}(\infty)})$ exists and 
\begin{align}
\# \lbrace (x,y)|~x\in \mathbb{U},~y\not \in \mathbb{U},~m\leq n,~m\in \mathbb{U},~m=x+y\rbrace \nonumber \\ \leq  \# \lbrace (x,y)|~x\in \mathbb{U},~y\not \in \mathbb{U},~m\leq n,~m\not \in \mathbb{U},~m=x+y\rbrace \nonumber   
\end{align}
with
\begin{align}
\#\lbrace (x,y)|~ \{x,y\} \cap \mathbb{U}\neq \emptyset,~m\leq n,~m\not \in \mathbb{U},~m=x+y\rbrace \ll \frac{n^{1-\epsilon}}{2} \nonumber
\end{align}
for some $\epsilon>0$, then
\begin{align}
  \mathcal{D}(\mathbb{U})=0.\nonumber
\end{align}
\end{theorem}

\begin{proof}
Let $\mathbb{U}\subset \mathbb{N}$ denote the sequence of Ulam numbers. Using Proposition \ref{inequality}, we obtain the lower bound 
\begin{align}
    \lim \limits_{n\longrightarrow \infty}\frac{\# \lbrace\mathbb{L}_{[x],[y]} \inn \mathcal{C}(n)|~ \{x,y\} \cap \mathbb{U}\neq \emptyset \rbrace}{\nu(n,\mathbb{N})}\geq \mathcal{D}(\mathbb{U}).\nonumber
\end{align}
Since $\nu(n,\mathbb{N})=\left \lfloor \frac{n-1}{2}\right \rfloor$, by the uniqueness of the axes of CoPs, we get 
\begin{align}
   \lim \limits_{n\longrightarrow \infty}\frac{\# \lbrace\mathbb{L}_{[x],[y]} \inn \mathcal{C}(n)|~ \{x,y\} \cap \mathbb{U}\neq \emptyset \rbrace}{\nu(n,\mathbb{N})}&=\lim \limits_{n\longrightarrow \infty}\frac{\# \lbrace\mathbb{L}_{[x],[y]} \inn \mathcal{C}(n,\mathbb{U})\rbrace}{\left \lfloor \frac{n-1}{2}\right \rfloor}\nonumber \\&+\lim \limits_{n\longrightarrow \infty}\frac{\# \lbrace\mathbb{L}_{[x],[y]} \inn \mathcal{C}(n)|~x\in \mathbb{U},y\in \mathbb{N}\setminus \mathbb{U}\rbrace}{\left \lfloor \frac{n-1}{2}\right \rfloor}\nonumber
\end{align}
since $\mathcal{D}(\mathbb{U}_{\mathcal{C}(\infty)})$ exists. We can now further write the following decomposition of the generators
\begin{align}
   \lim \limits_{n\longrightarrow \infty}\frac{\# \lbrace\mathbb{L}_{[x],[y]} \inn \mathcal{C}(n,\mathbb{U})\rbrace}{\left \lfloor \frac{n-1}{2}\right \rfloor}&= \lim \limits_{n\longrightarrow \infty}\frac{\# \lbrace\mathbb{L}_{[x],[y]} \inn \mathcal{C}(m,\mathbb{U})|~m\leq n,~m\in \mathbb{U}\rbrace}{\left \lfloor \frac{n-1}{2}\right \rfloor}\nonumber \\&+\lim \limits_{n\longrightarrow \infty}\frac{\# \lbrace\mathbb{L}_{[x],[y]} \inn \mathcal{C}(m,\mathbb{U})|~m\leq n,~m\not \in \mathbb{U}\rbrace}{\left \lfloor \frac{n-1}{2}\right \rfloor}.\label{ulam key}
\end{align}
It follows from the properties of the Ulam sequence the following reduction
\begin{align}
  \lim \limits_{n\longrightarrow \infty}\frac{\# \lbrace\mathbb{L}_{[x],[y]} \inn \mathcal{C}(m,\mathbb{U})|~m\leq n,~m\in \mathbb{U}\rbrace}{\left \lfloor \frac{n-1}{2}\right \rfloor}&=  \lim \limits_{n\longrightarrow \infty}\frac{1}{\left \lfloor \frac{n-1}{2}\right \rfloor}=0\nonumber
\end{align}
since $\# \lbrace\mathbb{L}_{[x],[y]} \inn \mathcal{C}(m,\mathbb{U})|~m\leq n,~m\in \mathbb{U}\rbrace=1$. Similarly, we have the decomposition 
\begin{align}
 \lim \limits_{n\longrightarrow \infty}\frac{\# \lbrace\mathbb{L}_{[x],[y]} \inn \mathcal{C}(n)|~x\in \mathbb{U},y\in \mathbb{N}\setminus \mathbb{U}\rbrace}{\left \lfloor \frac{n-1}{2}\right \rfloor}\nonumber \\= \lim \limits_{n\longrightarrow \infty}\frac{\# \lbrace\mathbb{L}_{[x],[y]} \inn \mathcal{C}(m)|~x\in \mathbb{U},y\in \mathbb{N}\setminus \mathbb{U},~m\leq n,~m\in \mathbb{U}\rbrace}{\left \lfloor \frac{n-1}{2}\right \rfloor}\nonumber \\+ \lim \limits_{n\longrightarrow \infty}\frac{\# \lbrace\mathbb{L}_{[x],[y]} \inn \mathcal{C}(m)|~x\in \mathbb{U},y\in \mathbb{N}\setminus \mathbb{U},m\leq n,~m\not \in \mathbb{U}\rbrace}{\left \lfloor \frac{n-1}{2}\right \rfloor}\nonumber \\\leq 2\lim \limits_{n\longrightarrow \infty}\frac{\# \lbrace\mathbb{L}_{[x],[y]} \inn \mathcal{C}(m)|~x\in \mathbb{U},y\in \mathbb{N}\setminus \mathbb{U},m\leq n,~m\not \in \mathbb{U}\rbrace}{\left \lfloor \frac{n-1}{2}\right \rfloor}\nonumber
\end{align}
by exploiting the condition
\begin{align}
    \# \lbrace\mathbb{L}_{[x],[y]} \inn \mathcal{C}(m)|~x\in \mathbb{U},y\in \mathbb{N}\setminus \mathbb{U},~m\in \mathbb{U},~m\leq n\rbrace \nonumber \\ \leq \# \lbrace\mathbb{L}_{[x],[y]} \inn \mathcal{C}(m)|~x\in \mathbb{U},y\in \mathbb{N}\setminus \mathbb{U},~m\not \in \mathbb{U},~m\leq n\rbrace \nonumber 
\end{align}
so that 
\begin{align}
 \# \lbrace\mathbb{L}_{[x],[y]} \inn \mathcal{C}(n)|~ \{x,y\} \cap \mathbb{U}\neq \emptyset \rbrace&= \# \lbrace\mathbb{L}_{[x],[y]} \inn \mathcal{C}(m,\mathbb{U})|~m\leq n,~m\not \in \mathbb{U}\rbrace \nonumber \\&+2\# \lbrace\mathbb{L}_{[x],[y]} \inn \mathcal{C}(m)|~x\in \mathbb{U},y\in \mathbb{N}\setminus \mathbb{U},m\leq n,~m\not \in \mathbb{U}\rbrace \nonumber 
\end{align}
and 
\begin{align}
    \mathcal{D}(\mathbb{U})&\leq \lim \limits_{m\longrightarrow \infty}\frac{\# \lbrace\mathbb{L}_{[x],[y]} \inn \mathcal{C}(m)|~ \{x,y\} \cap \mathbb{U}\neq \emptyset,~m\leq n,~m\not \in \mathbb{U} \rbrace}{\left \lfloor \frac{n-1}{2}\right \rfloor}\nonumber \\&+\lim \limits_{m\longrightarrow \infty}\frac{\# \lbrace\mathbb{L}_{[x],[y]} \inn \mathcal{C}(m)|~x\in \mathbb{U},~y\not \in \mathbb{U},~m\leq n,~m\not \in \mathbb{U} \rbrace}{\left \lfloor \frac{n-1}{2}\right \rfloor}.\nonumber
\end{align}
By exploiting the condition 
\begin{align}
 \# \lbrace (x,y)|~ \{x,y\} \cap \mathbb{U}\neq \emptyset,~m\leq n,~m\not \in \mathbb{U},~m=x+y,~x<y\rbrace \nonumber \\ \ll \frac{n^{1-\epsilon}}{2} \nonumber
\end{align}
for some $\epsilon>0$, we have 
\begin{align}
 \mathcal{D}(\mathbb{U})&\leq \lim \limits_{m\longrightarrow \infty}\frac{\# \lbrace\mathbb{L}_{[x],[y]} \inn \mathcal{C}(m)|~\{x,y\} \cap \mathbb{U}\neq \emptyset,~m\leq n,~m\not \in \mathbb{U} \rbrace}{\left \lfloor \frac{n-1}{2}\right \rfloor}\nonumber \\&+\lim \limits_{m\longrightarrow \infty}\frac{\# \lbrace\mathbb{L}_{[x],[y]} \inn \mathcal{C}(m)|~x\in \mathbb{U},~y\not \in \mathbb{U},~m\leq n,~m\not \in \mathbb{U} \rbrace}{\left \lfloor \frac{n-1}{2}\right \rfloor}\nonumber \\&\ll \lim \limits_{n\longrightarrow \infty} \frac{1}{n^{\epsilon}}=0.\nonumber  
\end{align}
\end{proof}

\bibliographystyle{amsplain}

\end{document}